\documentclass[12pt]{amsart}

\newtheorem{theorem}{Theorem}[section]
\newtheorem{lemma}[theorem]{Lemma}
\newtheorem{corollary}[theorem]{Corollary}
\newtheorem{proposition}[theorem]{Proposition}
\newtheorem{remark}[theorem]{Remark}
\newtheorem{definition}[theorem]{Definition}

\newtheorem{Question}[theorem]{Question}
\newcommand{\ncom}{\newcommand}
\ncom{\rar}{\rightarrow}
\ncom{\lrar}{\longrightarrow}
\ncom{\ov}{\overline}
\ncom{\m}{\mbox}
\ncom{\sta}{\stackrel}
\ncom{\comx}{{\mathbb C}}
\ncom{\Z}{{\mathbb Z}}
\ncom{\Q}{{\mathbb Q}}
\ncom{\R}{{\mathbb R}}
\ncom{\G}{{\mathbb G}}
\ncom{\al}{\alpha}
\ncom{\p}{{\mathbb P}}
\ncom{\E}{{\mathbb E}}
\ncom{\N}{{\mathbb N}}
\ncom{\K}{{\mathbb K}}
\ncom{\Le}{{\mathbb L}}
\ncom{\A}{{\mathbb A}}
\ncom{\B}{{\mathbb B}}
\ncom{\F}{{\mathbb F}}
\ncom{\C}{{\mathbb C}}
\ncom{\f}{\frac}
\ncom{\cA}{{\mathcal A}}
\ncom{\cX}{{\mathcal X}}
\ncom{\cO}{{\mathcal O}}
\ncom{\cW}{{\mathcal W}}
\ncom{\cL}{{\mathcal L}}
\ncom{\cP}{{\mathcal P}}
\ncom{\cH}{{\mathcal H}}
\ncom{\cS}{{\mathcal S}}
\ncom{\cM}{{\mathcal M}}
\ncom{\cC}{{\mathcal C}}
\ncom{\cT}{{\mathcal T}}
\ncom{\cF}{{\mathcal F}}
\ncom{\cN}{{\mathcal N}}
\ncom{\cJ}{{\mathcal J}}
\ncom{\cV}{{\mathcal V}}
\ncom{\cZ}{{\mathcal Z}}
\ncom{\cU}{{\mathcal U}}
\ncom{\cSU}{{\mathcal S \mathcal U}}
\ncom{\cG}{{\mathcal G}}
\ncom{\cQ}{{\mathcal Q}}
\ncom{\cR}{{\mathcal R}}
\ncom{\cE}{{\mathcal E}}
\ncom{\cY}{{\mathcal Y}}
\input xy
\xyoption{all}

\begin{document}
\baselineskip=16pt

\title[Rationality and Chow--K\"unneth decompositions]{Rationality and
Chow--K\"unneth decompositions for some moduli stacks of curves    }

\author[J. N. Iyer]{Jaya NN Iyer}
\author[S. M\"uller--Stach]{Stefan M\"uller--Stach}
\address{The Institute of Mathematical Sciences, CIT
Campus, Taramani, Chennai 600113, India}
\address{Department of Mathematics and Statistics, University of Hyderabad, Gachibowli, Central University P O, Hyderabad-500046, India}
\email{jniyer@imsc.res.in}

\address{Mathematisches Institut der Johannes Gutenberg University\"at Mainz,
Staudingerweg 9, 55099 Mainz, Germany}
\email{mueller-stach@uni-mainz.de}

\footnotetext{This work is partly supported by 
Sonderforschungsbereich/Transregio 45.}

\footnotetext{Mathematics Classification Number: 14C25, 14D05, 14D20,
 14D21}

\begin{abstract} In this paper we show the existence of a Chow--K\"unneth decomposition for the moduli
stack of stable curves with marked points $\ov\cM_{g,r}$, for low values of $g,r$.
We also look at the moduli space $\cR_{3,2}$ of double covers of genus three curves, branched along $4$ distinct points.
We first obtain a birational model of $\cR_{3,2}$ as a group quotient of a product of two Grassmannian varieties. This provides a Chow--K\"unneth decomposition over an open subset of $\cR_{3,2}$. The question of rationalilty of $\cR_{3,2}$ is also discussed.
\end{abstract}
\maketitle


\setcounter{tocdepth}{1}
\tableofcontents

\section{Introduction}
Suppose $X$ is a smooth projective variety of dimension $d$ over the complex numbers.
Let $CH^i(X)_\Q:=CH^i(X)\otimes \Q$ denote the rational Chow group of codimension $i$ algebraic cycles on $X$. One of the
important questions in the theory of algebraic cycles is to determine the structure of the rational Chow groups of $X$.
A conjecture of J. Murre \cite{Mu2}, \cite{Mu3} says
that the diagonal cycle $\Delta_X\subset X\times X$ has a splitting:
$$
\Delta_X=\oplus_{i=0}^{2d} \pi_i \,\in\,CH^d(X\times X)_\Q.
$$
Here $\pi_i$ are orthogonal projectors, for a ring structure on $CH^d(X\times X)_\Q$, and which lift the K\"unneth components of $\Delta_X$ in the rational Betti cohomology, see \S \ref{CK}. A decomposition as above is shown to yield a filtration of the rational Chow group by J. Murre.

The cases where a decomposition as above holds include curves, surfaces, uniruled threefolds, abelian varieties and some varieties with a nef tangent bundle \cite{Mu1}, \cite{Sh}, \cite{dA-Mu}, \cite{dA-Mu2},  \cite{Iy}. Some universal families over Shimura surfaces or other varieties have also been investigated, and a Chow--K\"unneth decomposition have been obtained in some cases, see
\cite{GHM2}, \cite{MM}.

In this paper, we continue our investigation on the Chow--K\"unneth decomposition for the moduli spaces of curves.
This is a sequel to \cite{Iy-Mu}, which included an introduction to the equivariant Chow motive for varieties with a group action. Since the moduli spaces $\cM_g$ for small $g\leq 9$ are known to be birationally isomorphic to a group quotient
of a homogeneous space, we could obtain a Chow--K\"unneth decomposition of $\cM_g$, at least over an open subset in the sense of definition \ref{CK-def}.
Here we would like to enlarge the class of examples, by looking at the moduli stacks of curves with marked points and also include the case of the stable moduli space. The other example is the moduli space of double covers of curves.
We investigate the moduli space $\cR_{3,2}$ \cite{BCV} of double covers of genus three curves branched over $4$ distinct
points, in some detail. The methods and results in this paper also extend to some other moduli spaces $\cR_{g,b}$, for small
$g$ and $b$.

Recall that the stable cohomology of $\cM_g$ is the part which is stable under various standard pullback maps, see \S \ref{stablecoh}. 
Alternatively, it is the cohomology of the limiting group $\Gamma_\infty$ of the various mapping class groups $\Gamma^s_g$, for a connected 
compact surface of genus $g$ and $s$ marked points  (for example see \cite{Ma-We}).

\begin{theorem} Suppose $\ov{\cM_{g,s}}$ denotes the moduli stack of stable curves of genus $g$ with $s$ marked points.
Then the following hold:

1) The stable rational cohomology has a Chow--K\"unneth decomposition.

2) The moduli stacks $\ov{\cM_{g,s}}$ have an explicit Chow--K\"unneth decomposition, if
$g=1, s\leq 3$ or $g=2, s\leq 2$ or  $g=3, s\leq 1$ or when $g=4, s=0$.
\end{theorem}

The key point used in the proof is that the moduli stacks as above have only algebraic cohomology. This enables us to construct orthogonal projectors.

The other class of examples include the moduli space of double covers. These spaces have attracted wide interest, with respect to the study of moduli of abelian varieties, and also regarding questions on rationality/unirationality.
We consider the moduli space $\cR_{3,2}$ studied by Bardelli-Ciliberto-Verra \cite{BCV}. This space parametrises data:
$(C,L,B)$, $C$ is a smooth connected projective curve of genus $3$, $L$ is a line bundle of degree $2$ on $C$ and $B$
is a divisor in the linear system $|L^2|$, consisting of distinct points. We first describe this moduli space as follows:

\begin{theorem}\label{uni}
The moduli space $\cR_{3,2}$ is birational to the group quotient of a product of Grassmannians $G(3,U^+)\times G(4,U^-)$, by a subgroup $H\subset SO(10)$.  Here $H$ is contained in the centraliser of the action of an involution $i$ on $SO(10)$. Moreover, there is an irreducible $16$-dimensional projective representation $U$ of $SO(10)$ and $U=U^+\oplus U^-$ is a splitting as $\pm$-eigenspaces for the involution $i$ acting on $U$.
We can write the birational equivalence as
$$
\cR_{3,2}\sim (G(3,U^+)\times G(4,U^-))/H.
$$
\end{theorem}

See \S \ref{birationalmodel}, for a proof.

As a consequence of this description, we obtain a Chow--K\"unneth decomposition for an open subset of $\cR_{3,2}$, see
Corollary \ref{CKopen}.

This description is similar to the descriptions obtained for the various moduli spaces $\cM_g$, for small $g\leq 9$, by Mukai and others (for example, see \cite{Mukai3}, \cite{Mukai6}).
The proof is by analyzing Mukai's description of the moduli space $\cM_7$ and restricting our attention to the sublocus
$\cR_{3,2}\subset \cM_7$. This sublocus is in the singular locus of $\cM_7$ and parametrises curves with an involution.
The involution plays a crucial role in determining the Grassmannian varieties, in the statement of Theorem \ref{uni}. We have been unable to determine explicitly the subgroup $H$ in the above theorem. This may be of independent interest and we pose this as a question, see \ref{question}.

In other direction, it is of wide interest to know when the moduli spaces are rational or unirational varieties. It is known from the results of Severi, Sernesi, Katsylo, Mukai, Dolgachev, Chang-Ran, Verra that the moduli spaces $\cM_g$, for small $g\leq 14$ are unirational \cite{Sernesi}, \cite{Katsylo}, \cite{Dolgachev}, \cite{Verra}, \cite{chen}.
Some moduli spaces of double covers have also been shown to be rational by Bardelli-Del Centina \cite{B-dC}, Izadi-Lo Giudice-Sankaran \cite{Izadi} and unirationality of $\cR_5$ is known. The above description
of $\cR_{3,2}$ in Theorem \ref{uni}, says that it is a unirational variety. We also describe $\cR_{3,2}$ as birationally equivalent to a $\p^1$-bundle over (an open subset) of the universal Picard scheme $\m{Pic}^2_{\cM_3}$ over the moduli space $\cM_3$. This gives us the following:

\begin{theorem}
The moduli space $\cR_{3,2}$ is a rational variety, if the variety $\m{Pic}^2_{\cM_3}$ is rational.
\end{theorem}

See Corollary \ref{doublerational}, and Remark \ref{remverra} for the question of rationality of  $\m{Pic}^2_{\cM_3}$. 

{\Small
Acknowledgements: This work is a sequel to \cite{Iy-Mu} (preprint Oct. 2007) which looked at the question of providing Chow--K\"unneth decomposition for some moduli spaces of curves of small genus. This required us to introduce equivariant Chow K\"unneth projectors and equivariant Chow motive. The first named author acknowledges and thanks the Women in Mathematics Program on 'Algebraic geometry and Group Actions' in May 2007, at IAS Princeton, where the equivariant cohomology theory and equivariant objects were discussed. She also thanks the Maths Department at Mainz, for their hospitality and support in June 2008 when this work was partly done. We also thank A. Verra for interesting communications and comments and B. Totaro for pointing out some errors and making suggestions. 
}

\section{Chow-K\"unneth decompositions for $\ov\cM_{g,r}$, for small $g,r$}

\subsection{Category of motives}\label{CK}
The category of nonsingular projective varieties over $\comx$ will be
denoted by $\cV$. Let $CH^i(X)_\Q:=CH^i(X)\otimes \Q$ denote the rational Chow group of
codimension $i$ algebraic cycles modulo rational equivalence. We look into the category of motives $\cM_\sim$, where $\sim$ is any adequate equivalence. For instance $\sim$ is homological or numerical equivalence.

Suppose $X$ is a smooth projective variety over $\comx$ of dimension $n$. Let $\Delta_X$ be the diagonal in $X\times X$. Consider the K\"unneth decomposition of $\Delta_X$ in the Betti cohomology:
$$
\Delta_X= \oplus_{i=0}^{2n}\pi_i^{hom}
$$
where $\pi_i^{hom}\in H^{2n-i}(X)\otimes H^i(X)$.

The motive of $X$ is said to have a \textit{K\"unneth decomposition} if each of
the classes $\pi_i^{hom}$ are algebraic, i.e., $\pi_i^{hom}$ is the image of an algebraic cycle $\pi_i$, which add to the diagonal cycle,
under the cycle class map from the rational Chow groups to the Betti cohomology.
Furthermore, the motive of $X$ is said to have a \textit{Chow--K\"unneth decomposition} if
each of the
classes $\pi_i^{hom}$ is algebraic and are orthogonal projectors,
 i.e.,
$\pi_i\circ \pi_j=\delta_{i,j}\pi_i$ and which add to the diagonal cycle $\Delta_X$ in $CH^n(X\times X)_\Q$. Here $\circ$ denotes the ring structure on
$CH^n(X\times X)_\Q$.

In \cite{Iy-Mu}, we showed explicit Chow--K\"unneth projectors for the universal curve over suitable open subsets of the moduli space of smooth curves $\cM_g$, when $g\leq 9$. In this paper, we are interested at looking at the Chow--K\"unneth decompositions for the moduli spaces of stable curves $\ov{\cM_{g,s}}$. These spaces are normal projective varieties and have singularities. So it is convenient to consider them as the Deligne--Mumford stacks (henceforth termed as DM-Stacks) which are smooth stacks. For this purpose, we define a Chow--K\"unneth decomposition for DM--stacks.

\subsection{Motives of Deligne--Mumford stacks}

Suppose $\cX$ is a smooth DM-stack with the projection $p:\cX\rar X$ to its coarse moduli space $X$.

Mumford, Gillet (\cite{Mumford},\cite{Gi}) have defined Chow groups for DM-stacks.
So from \cite[Theorem 6.8]{Gi}, the pullback $p^*$ and pushforward maps $p_*$ establish a
ring isomorphism of the rational Chow groups

\begin{equation}\label{ringiso}
CH^{\ast}(\cX)_{\Q }\cong CH^{\ast}(X)_{\Q }.
\end{equation}

This can be applied to the product $p\times p:\cX\times \cX\rar X\times X$, to get a ring isomorphism
\begin{equation}\label{ringiso2}
CH^{\ast}(\cX\times \cX)_{\Q }\cong CH^{\ast}(X\times X)_{\Q }.
\end{equation}

These isomorphisms also hold in the rational singular cohomology of $\cX$ and $\cX\times \cX$ (for example see \cite{Behrend}):
\begin{equation}\label{cringiso}
H^{\ast}(\cX,\Q ) \cong H^{\ast}(X,\Q ).
\end{equation}
\begin{equation}\label{cringiso2}
H^{\ast}(\cX\times \cX,\Q )\cong H^{\ast}(X\times X,\Q ).
\end{equation}
 Assume that $X$ is a projective variety of dimension $n$. Via the isomorphisms in the cohomology, we can pullback the K\"unneth decomposition of $\Delta_X$ in $H^{2n}(X\times X,\Q)$ to a decomposition
of $\Delta_{\cX}$ in $H^{2n}(\cX\times \cX,\Q)$, whose components we refer to as the K\"unneth components of $\cX$.

Consider the diagonal substack $\Delta_\cX$ in $\cX\times \cX$. Then we can write
$$
\Delta_\cX = \oplus_{i=0}^{2n}\pi_i^{hom}
$$
where $\pi_i^{hom}\in H^{2n-i}(\cX)\otimes H^i(\cX)$.

The motive of $\cX$ is said to have a \textit{K\"unneth decomposition} if each of
the classes $\pi_i^{hom}$ are algebraic, i.e., $\pi_i^{hom}$ is the
 image of an algebraic cycle $\pi_i$ which add to the Chow diagonal cycle,
under the cycle class map from the rational Chow groups to the Betti
 cohomology of $\cX\times \cX$.
Furthermore, the motive of $X$ is said to have a \textit{Chow--K\"unneth decomposition} if
 each of the
classes $\pi_i^{hom}$ is algebraic and are orthogonal projectors,
 i.e.,
$\pi_i\circ \pi_j=\delta_{i,j}\pi_i$, which add to the diagonal cycle $\Delta_\cX$ in $CH^n(\cX\times \cX)_\Q$ . Here $\circ$ is the ring structure on
$CH^n(\cX\times \cX)_\Q$, defined in the same way when $\cX$ is a smooth projective variety.
We extend the notion of orthogonal projectors on a smooth stack, as follows.

\begin{definition}\label{CK-def}
Suppose $\cX$ is a smooth DM-stack with a quasi-projective coarse moduli space $X$.
The motive $(\cX,\Delta_\cX)$ of $\cX$ is said to have
a $\textbf{K\"unneth decomposition}$ if the classes
 $\pi_i$ are
algebraic, i.e., they have a lift in the Chow group
 $CH^n(\cX\times \cX)_\Q$ and add to the Chow diagonal class. Furthermore, if $\cX$ admits a smooth compactification
 $\cX\subset \ov{\cX}$ such that the K\"unneth projectors extend
to orthogonal projectors on $\ov{\cX}$ then we say that $\cX$ has a
 \textbf{Chow--K\"unneth decomposition}.
\end{definition}

We recall the following lemma from \cite{Iy-Mu}, which also applies for smooth stacks.

\begin{lemma}\label{simpleprojectors}
Suppose $\cY$ is a smooth DM-stack whose coarse moduli space is projective of dimension $n$ over
 $\comx$. Let $H^*(\cA)$ be the subalgebra of the cohomology algebra $H^*(\cY,\Q)$
consisting of only algebraic classes. Denote the graded pieces of $H^*(\cA)$ by $H^i(\cA)$ for all $0 \le i \le m$, for some $m < n$.
Then we can construct orthogonal projectors
$$
\pi_0,\pi_1,...,\pi_m,\pi_{2n-m},\pi_{2n-m+1},...,\pi_{2n}
$$
in the usual Chow group $CH^n(\cY\times \cY)_\Q$, and where $\pi_{2i}$ acts
 as $\delta_{i,p}$ on $H^{2p}(\cA)$ and $\pi_{2i-1}=0$.
\end{lemma}
\begin{proof}
See \cite[Lemma 5.2]{Iy-Mu}, when $H^*(\cA)=H^*(\cY,\Q)$. The same arguments also hold for the subalgebra $H^*(\cA)=\oplus_{p\geq 0} H^{2p}(\cA)$. Indeed, let $H^{2p}(\cA)$ be
 generated by cohomology classes of cycles $C_1,\ldots,C_s$ and $H^{2r-2p}(\cA)$
  be generated by cohomology classes of cycles $D_1,\ldots,D_s$. We
 denote by $M$ the intersection matrix with entries
$$
M_{ij}= C_i \cdot D_j \in \Z.
$$ 
After base change and passing to $\Q$--coefficients we may assume that
 $M$ is diagonal, 
since the cup--product $H^{2p}(\cA) \otimes H^{2r-2p}(\cA) \to \Q$ is
 non--degenerate. We define the projector $\pi_{2p}$ as 
$$
\pi_{2p}=\sum_{k=1}^s \frac{1}{M_{kk}} D_k \times C_k. 
$$
It is easy to check that $\pi_{2p\,*}(C_k)=D_k$. Define $\pi_{2r-2p}$
 as the adjoint, i.e., transpose of $\pi_{2p}$.
Via the Gram--Schmidt process from linear algebra we can successively
 make all projectors orthogonal.
\end{proof}

\subsection{The stable cohomology of ${\cM_{g}}$}\label{stablecoh}

In this subsection, we recall some results on the stable cohomology of the moduli spaces \cite{Ha},\cite{Lo}. Our aim will be to show the existence of a Chow--K\"unneth decomposition for the stable cohomology.

Denote $S_g$,  compact connected oriented surface of genus $g$ with $s$ marked points.
Let $\Gamma^s_g$ denote the mapping class group, the connected component of identity of the group of orientation preserving diffeomorphisms of $S_g$.
J. Harer \cite{Ha} has proved a stability theorem which essentially says that the cohomology group $H^k(\Gamma^s_g,\Z)$ only depends on $s$ if $g$ is large compared to $s$. We would like to state it in more geometric terms and fix some notations below. Denote $u_i\in H^2(\Gamma^s_g,\Z)$ for the first Chern class.

Fix a finite ordered set $S$ of cardinality $s$. We denote by $\cC^S_g$ the moduli space of pairs $(C,x)$ where $C$ is a compact Riemann surface of genus $g$ and $x:S\rar C$ is a map. Let $j:\cM^S_g\subset \cC^S_G$ be the open subset defined by the condition that $x$ be injective. In other words, $\cM^S_g$ is the moduli space   of smooth curves with $s$ marked points.

Now $\cM^S_g$ (resp. $\cM_g$) is a virtual classifying space of $\Gamma^s_g$ (resp. $\Gamma_g$). In particular $\Gamma^s_g$ and $\cM^S_g$ have the same rational cohomology. Let $\cC_g$ be the universal curve and denote by $\theta$ its relative tangent sheaf. For each $i\in S$, the map $(C,x)\mapsto x(i)$ defines a projection $\cC^S\rar \cC_g$; denote by $\theta_i$ the pullback of $\theta$ under this map.

\begin{proposition}
The ring homomorphism
$$
\psi^S_g:H^\bullet(\cM_g,\Q)[u_i:i\in S]\rar H^\bullet(\cM^S_g,\Q),\,u_i\mapsto c_1(\theta_i)_{|\cM^S_g}
$$
is an isomorphism in degree $\leq N(g)$.
\end{proposition}
\begin{proof}
See \cite[Proposition 2.2]{Lo}.
\end{proof}

Similarly, the rational cohomology of $\cC_g^S$ is expressed in terms of that of $\cM_g$, $u_i$ and the multi-diagonal classes $\cC_g(P_I)$, where $P_I$ is a partition of $S$ whose parts are $I$ and the singletons in $S-I$.

  More formally, we consider the $\Q[u_i:i\in S]$-algebra $A^\bullet_S$ generated by $a_I$, where $I$ runs over the subsets of $S$ with at least two elements.
These generators satisfy the relations
\begin{eqnarray*}
u_ia_I & := & u_ja_I, \mbox{ if }i,j\in I \\
a_Ia_J & := & u_i^{|I\cap J|-1}, \mbox{ if } i\in I\cap J.
\end{eqnarray*}
For every partition $P$ of $S$ put $a_P:=\Pi_{I\in P:|I|\geq 2}a_I$, with the convention that $a_P=1$ if $P$ is a partition into singletons.

Then we have
\begin{proposition}
There is an algebra homomorphism
$$
\phi^S_g:H^\bullet(\cM_g,\Q)\otimes A^\bullet_S\rar H^\bullet(\cC^S_g,\Q)
$$
that extends the natural homomorphism $H^\bullet(\cM_g,\Q)\rar H^\bullet(\cC^S_g,Q)$, and sends $1\otimes u_i\mapsto c_1(\theta_i),\,1\otimes a_I\mapsto \cC_g(P_I)$ (the Poincar\'e dual class). This is a $\Sigma_s$-equivariant homomorphism, and a morphism of mixed Hodge structures. Moreover, $\phi^S_g$ is an isomorphism
in degrees $\leq N(g)$.
\end{proposition}
\begin{proof}
See \cite[Theorem 2.3]{Lo}.
\end{proof}

Here $N(g_0)$ is the maximal integer such that $\phi^S_g$ induces isomorphisms
in degrees $\leq N(g_0)$ for all $g\geq g_0$ and $s\geq 0$. Some bounds on $N(g)$ are given in \cite{Ha}, \cite{ivanov}, \cite{ivanov2}.

The \textit{stable cohomology} of $\cM_g$ is the cohomology space for which the maps
$\psi^S_g,\phi^S_g$ are isomorphisms (alternately, it is the cohomology of the
limiting group $\Gamma_\infty$ of the various $\Gamma^s_g$, or the  rational cohomology of the stable moduli space, for example see
\cite{Ma-We}).

I. Madsen and M. Weiss \cite{Ma-We} have proved Mumford's conjecture on the structure of the stable cohomology space of $\cM_g$.

\begin{theorem}\cite{Ma-We}\label{Mumfordconjecture}
The stable cohomology of $\cM_g$ is generated by the classes $\kappa_i$. Here $\kappa_i$ are canonical algebraic classes defined by Mumford in \cite{Mumford}. The class $\kappa_i$ is the direct image of the $i+1$-st power of the first Chern class of the relative dualizing sheaf of $\cC_g\rar \cM_g$.
\end{theorem}

\begin{corollary}
The (virtual) stable moduli space $\cM_g$ has a Chow--K\"unneth decomposition in the sense of definition \ref{CK-def}.
\end{corollary}
\begin{proof}
By Theorem \ref{Mumfordconjecture}, the stable cohomology is generated by algebraic classes. We can now apply Lemma \ref{simpleprojectors} to get algebraic K\"unneth projectors. These projectors can be extended to orthogonal projectors in the smooth compactification $\ov\cM_g$. Indeed, we can take the natural closure of the cycles $\kappa_i$ on the DM-stack $\ov\cM_g$ and take the $\Q$-subalgebra generated by these classes in $H^\bullet(\ov\cM_g,\Q)$. Then applying Lemma \ref{simpleprojectors} to this $\Q$-subalgebra, orthogonal projectors can be defined which restrict to the K\"unneth projectors on $\cM_g$.
\end{proof}

\subsection{Chow--K\"unneth decomposition for the moduli stack $\ov{\cM_{g,s}}$}

In this subsection, we will look at the DM--compactified moduli stacks $\ov{\cM_{g,s}}$ and show the existence of the Chow--K\"unneth decomposition when $g,s$ are small.

We recall the following results on $\ov{\cM_{g,s}}$.

\begin{theorem}\label{algcoh}
The rational cohomology of the moduli stack $\ov{\cM_{g,s}}$ has no odd cohomology and is generated by algebraic classes, if
$g=1, s\leq 3$ or $g=2, s\leq 2$ or  $g=3, s\leq 1$ or when $g=4, s=0$.
\end{theorem}
\begin{proof}
When $g\leq 2$, see \cite{getzler}.

When $g=3$, see \cite{looijenga2}, \cite{getzler-looijenga}.

When $g=4$, see \cite{tommasi}.
\end{proof}

\begin{corollary}
The moduli stacks $\ov{\cM_{g,s}}$ have an explicit Chow--K\"unneth decomposition, if
$g=1, s\leq 3$ or $g=2, s\leq 2$ or  $g=3, s\leq 1$ or when $g=4, s=0$.
For any $g$ and $s$, one can always construct canonical orthogonal projectors 
$$
\pi_0,\pi_1,\pi_2,\pi_3,\pi_5, \pi_{2n-5},\pi_{2n-3},\pi_{2n-2},\pi_{2n-1}, \pi_{2n}
$$
 where $n:=\m{dim}\ov{\cM_{g,s}}$.
\end{corollary}
\begin{proof}
The first assertion follows from Lemma \ref{simpleprojectors} and Theorem \ref{algcoh}. The second assertion follows from the fact that $H^2(\ov{\cM_{g,s}},\Q)$
is always algebraic and $H^i(\ov{\cM_{g,s}},\Q)=0$ if $i=1,3,5$ (see \cite{cornalba}).
\end{proof}


\section{A birational model of the moduli space $\cR_{3,2}$ of Bardelli-Ciliberto-Verra}\label{birationalmodel}


In this section, we will look into the question of describing the moduli space
$\cR_{3,2}$ studied by Bardelli-Ciliberto-Verra \cite{BCV}. The description is similar to the description of the moduli space $\cM_g$, for small $g$, studied by several authors (for example see \cite{Mukai3}, \cite{Mukai6}, \cite{Dolgachev}). The birational model is usually a group quotient of a homogeneous space. Such a description is useful in exhibiting a Chow--K\"unneth decomposition at least
over an open subset of the moduli space and also for addressing the question of rationality/unirationality. This was used in \cite{Iy-Mu} for obtaining a Chow--K\"unneth decomposition
of open subsets of $\cM_g$, for $g\leq 9$. We would like to extend similar conclusions to open subsets of the moduli space of double covers. This will be carried out in the next section.

More precisely, let $\cR_{3,2}$ be the moduli space of all isomorphism classes of double coverings $f:C'\rar C$ with $C$ a smooth curve of genus $3$, $C'$ irreducible and $f$ is branched at $4$ distinct points of $C$. Alternately, $\cR_{3,2}$ is the moduli space of isomorphism classes of triples $(C,B,\cL)$, where $C$ is a smooth curve of genus $3$, $B$ is an effective divisor on $C$ formed by $4$ distinct points, $\cL$ is a line bundle on $C$ such that $\cL^{\otimes 2}\simeq \cO(B)$.

Note that the genus of the curve $C'$ is $g'=7$ and $\cR_{3,2}\subset \cM_{7}$.
Then we have
$$
\mbox{dim}\cR_{3,2}= 10.
$$
Our main theorem in this section is the following:

\begin{theorem}\label{birational}

The moduli space $\cR_{3,2}$ is birational to the group quotient of a product of Grassmannians $G(3,U^+)\times G(4,U^-)$, by an algebraic subgroup 
$H\subset SO(10)$.  Here $H$ is contained in the centraliser of the action of an involution $i$ on $SO(10)$. Moreover, there is an irreducible $16$-dimensional projective representation $U$ of $SO(10)$ and $U=U^+\oplus U^-$ is a splitting as $\pm$-eigenspaces for the involution $i$ acting on $U$.
We can write the birational equivalence as
$$
\cR_{3,2}\sim (G(3,U^+)\times G(4,U^-))/H.
$$
\end{theorem}

Our proof follows by analysing  Mukai's classification \cite{Mukai3}, \cite{Mukai6} of the generic genus $7$ canonical curve, taking into account the action of the involution. Whenever a genus $7$ smooth curve is not tetragonal, then it is a linear section of an orthogonal Grassmannian $X_{10}\subset \p^{15}$, given by the spinor embedding (see \cite[p.1632]{Mukai3}). Here $\p^{15}=\p(U_{16})$ where $U_{16}$ is the irreducible spinor representation of the spin group $Spin(10)$. Hence the space $U_{16}$ is a projective representation of the special orthogonal group $SO(10)$. Projectively, this can be translated to say that the group $SO(10)$ acts on $\p^{15}$ and leaves the orthogonal Grassmannian $X_{10}$ invariant. In particular $SO(10)$ also acts on the linear subspaces of $\p^{15}$ and we will require its action on the Grassmannian $G(7,U_{16})$.
This is because a general linear subspace $\p^6\subset \p^{15}$ restricted to $X_{10}$
gives a canonical curve $C$ of genus $7$. In other words, $\p^6$ is the complete linear system given by the canonical bundle on $C=\p^6\cap X_{10}$.

Furthermore, we have the following result on the embedding into the homogeneous space.

\begin{theorem}\label{automorphism}
Assume that two linear spaces $P_1,P_2$ cut out smooth curves $C_1, C_2$ from the symmetric space $X_{10}\subset \p^{15}$ respectively. Then any isomorphism
from $C_1$ onto $C_2$ extends to an automorphism $\phi$ of $X_{10}\subset \p^{15}$ with $\phi(P_1)=P_2$.
\end{theorem}
\begin{proof}
See \cite[Theorem 3]{Mukai6}.
\end{proof}
This theorem characterises the non-tetragonal curves of genus $7$.
Explicitly, the moduli space has the following birational model \cite[\S 5, p.1639]{Mukai3}:
\begin{eqnarray*}
\cM_7 & \sim & G(7,U^{16})/SO(10).\\
\end{eqnarray*}

To obtain a birational model of $\cR_{3,2}$, we will utilise the above birational model of $\cM_7$ and analyse the birational equivalence restricted
to the sublocus $\cR_{3,2}$.

We will need the following lemma in our proof of Theorem \ref{birational}.
We say that a curve $C'$ is tetragonal if and only if there is a line bundle $L\in g^1_4(C')$.

\begin{lemma}\label{nontetragonal}
Consider a double cover $f:C'\rar C$, defined by a line bundle $\cL$ branched along the set $B$ of $4$ distinct points, and such that $\cL^2=\cO(B)$. Assume that $C, C'$ are not hyperelliptic.
The curve $C'$ has a $L\in g^1_4$ only if $L$ is the pullback of a line bundle of degree $2$ on $C$.
\end{lemma}
\begin{proof}
The arguments are similar to \cite[Proposition 2.5, p.234]{Ramanan}, and we explain them below.
Let $L\in g^1_4(C')$, i.e., $L$ is a line bundle of degre $4$ on $C'$ and $h^0(L)= 2$. If $L\simeq i^*L$ then $L$ descends down to the quotient curve $C$ as a line bundle of degree $2$. Suppose $L$ is not isomorphic to $i^*L$.
Consider the evaluation sequence:
$$
0\rar N \rar H^0(L)\otimes \cO_{C'} \rar L\rar 0.
$$
Since $h^0(L)=2$ we see that $N\simeq L^{-1}$.
Tensor the above exact sequence by $i^*L$ and take its global sections. Since $L\neq i^*L$, we observe that $H^0(N\otimes i^*L)=0$ and hence $H^0(L)\otimes H^0(i^*L)\subset H^0(L\otimes i^*L)$. In particular,
$h^0(L\otimes i^*L)\geq 4$. Since $C'$ is non-hyperelliptic, by Clifford's theorem \cite[IV,5.4]{Arbarello},
$h^0(L\otimes i^*L)\leq 4$. Hence we obtain the equality $H^0(L)\otimes H^0(i^*L)\,=\, H^0(L\otimes i^*L)$.

Now, notice that the line bundle $L\otimes i^*L$ has degree $8$ on $C'$ and is invariant under $i$. Hence the product line bundle descends down to $C$ as a line bundle of degree $4$, call this line bundle $M$. In other words, $L\otimes i^*L\simeq f^*M$.
Consider the direct image
$$
f_\ast(\cO_{C'})= \cO_C\oplus \cL^{-1}.
$$
Hence, by projection formula, $f_*(L\otimes i^*L)= M\oplus (M\otimes \cL^{-1})$.
This gives a decomposition 
$$
H^0(C',L\otimes i^*L)=H^0(C,M)\oplus H^0(C,M\otimes \cL^{-1}).
$$
Moreover, we can identify the eigenspaces for the involution $i$ as follows:
\begin{equation}\label{eigen}
H^0(C',L\otimes i^*L)^+\,=\,H^0(C,M),\,\, H^0(C',L\otimes i^*L)^-=H^0(C,M\otimes \cL^{-1}).
\end{equation}
By Riemann-Roch applied to $M$ and $M\otimes \cL^{-1}$ on $C$, we get the dimension counts: $h^0(M)=3$ if $M=\omega_C$, otherwise $h^0(M)=2$. Furthermore, since $C$ is non-hyperelliptic 
\begin{equation}\label{clifford}
h^0(M\otimes \cL^{-1})= 0.
\end{equation} 
by Clifford's theorem and Riemann-Roch.
This implies that
\begin{equation}\label{equidim}
H^0(L)\otimes H^0(i^*L)=H^0(L\otimes i^*L)=H^0(f^*M)=H^0(M).
\end{equation}
The first equality in \eqref{equidim} implies that the $\pm$-eigenspaces for the involution $i$ are non-zero. This gives a contradiction to \eqref{eigen} and \eqref{clifford}.

\end{proof}

\begin{corollary}\label{generic}
The generic curve in $\cR_{3,2}$ is non-tetragonal.
\end{corollary}
\begin{proof}
By formula~\eqref{clifford} in the proof of Lemma \ref{nontetragonal}, the generic line bundle $M$ of degree $2$ on a generic curve of genus $3$ has no section. The eigenspace decomposition for the sections of the pullback bundle $L:=f^*M$ is given as
$$
H^0(C',L)=H^0(C,M)\oplus H^0(C,M\otimes \cL^{-1}).
$$ 
and which implies that the generic curve $(C',\cL,B)$ in $\cR_{3,2}$ is a non-tetragonal curve. 

\end{proof}

\subsection{Proof of Theorem \ref{birational}}
Consider the inclusion $\cR_{3,2}\subset \cM_7$ of moduli spaces.
Then we recall the classification of the singular loci of the moduli space $\cM_g$ done by Cornalba \cite{Cornalba2}. In particular, the curves with non-trivial automorphisms lie in the singular loci of $\cM_g$ and precisely form the singular loci. The maximal components of the singular loci are also described by him.
We recall his result when $g=7$ and for the embedding $\cR_{3,2}\subset \cM_7$, since it will be crucial for us. We note that any double cover $(C'\rar C)\in \cR_{3,2}$ corresponds to an involution $i$ on $C'$ with four fixed points, and having the quotient $C=C'/i$.

\begin{proposition}\label{cornalba2}
The singular loci $\cS\subset \cM_7$ consist of smooth curves with automorphisms. 
In particular the moduli space $\cR_{3,2}$ lies in the singular loci $\cS$ and furthermore it is a maximal component of $\cS$.
\end{proposition}
\begin{proof}
See \cite[Corollary 1, p.146 and p.150]{Cornalba2}.
\end{proof}

Now, consider a generic point $(C'\sta{f}{\rar} C)=(C,B,\cL)\in \cR_{3,2}$. Then, by \cite[\S 2]{BCV}, we have a decomposition of the canonical space of $C'$:
\begin{equation}\label{eigenspaces}
H^0(C',\omega_{C'})= H^0(C,\omega_{C})\oplus H^0(C,\omega_{C}\otimes \cL).
\end{equation}
We can also interpret this decomposition for the involution $i$, which acts on the canonical space nontrivially.
Namely, we have a natural identification of the eigenspaces for $i$:
\begin{eqnarray*}
H^0(C',\omega_{C'})^+ & = & H^0(C,\omega_{C}) \\
H^0(C',\omega_{C'})^- & = & H^0(C,\omega_{C}\otimes \cL).
\end{eqnarray*}
Note that $\m{dim }H^0(C',\omega_{C'})^+=3$ and  $\m{dim }H^0(C',\omega_{C'})^-=4$.

We can now apply Theorem \ref{automorphism}, to the automorphism $i$ and conclude that $i$ lifts to an automorphism $i$ of $\p^{15}$ and leaves $X_{10}$ invariant. This gives an action of $i$ on the representation space $U_{16}$. Indeed, since $\m{Pic}(X_{10})\simeq \Z$, the ample line bundle $\cO_{X_{10}}(1)$ is invariant under $i$ and hence induces an action on its  sections which is precisely $U_{16}$.
Let us write the eigenspace decomposition of $U_{16}$ for the $i$-action:
\begin{equation}\label{pmeigenspace}
U_{16}\,=\,U^+ \oplus U^-.
\end{equation}

There are various possibilities for the dimensions of $U^+$ and $U^-$, which will
correspond to
\begin{equation}\label{dimension}
(\m{dim }U^+,\m{dim }U^-):=(r,16-r), \m{ for } 1\leq r \leq 15,
\end{equation}
since $i$ acts nontrivially.

We make the following observation first.
\begin{lemma}\label{invariant}
A point of the product variety $G(3,U^+)\times G(4,U^-)\subset G(7,U_{16})$ corresponds to a linear space
$\p^6\subset \p^{15}$, which is $i$ invariant. Furthermore, if $\p^6$ intersects
$X_{10}$ transversely then the intersection is a non-tetragonal curve with an involution and satisfying the decomposition \eqref{eigenspaces}.
\end{lemma}
\begin{proof}
We first note that a $3$-dimensional subspace $V^+\subset U^+$ and $4$-dimensional subspace $V^-\subset U^-$, gives a linear subspace $\p^6\subset \p^{15}$.
Clearly $\p(V^+\oplus V^-)\subset \p(U_{16})$ is a $\p^6$ and is invariant under the action of $i$. For the second assertion, note that $C'=\p^6\cap X_{10}$
also is an $i$-invariant subset and whenever the intersection is transverse, it correspond to a genus $7$ curve $C'$(by \cite{Mukai3}) with an involution, such that $\p^6$ is the canonical linear system of $C'$. This means that the $\pm$-eigenspaces of the canonical space of $C'$ are precisely $V^+$ and $V^-$. These data recover the decomposition in \eqref{eigenspaces}.
\end{proof}

\begin{lemma}\label{SO}
There is a subgroup $H\subset SO(10)$ such that $U^+$ and $U^-$ are $H$-representations. This induces an action of $H$ on $G(3,U^+)\times G(4,U^-)$ and which commutes with the action of $i$ such that the group quotient under this action is a birational model of $\cR_{3,2}$. In other words, we can write
$$
\cR_{3,2}\sim (G(3,U^+)\times G(4,U^-))/H.
$$
\end{lemma}
\begin{proof}
We note that by Mukai's classification \cite[\S 5]{Mukai3}, we have a birational isomorphism
$$
\cM_7  \sim  G(7,U^{16})/SO(10).
$$
The product subvariety $G(3,U^+)\times G(4,U^-)\subset G(7,U^{16})$ is not acted by $SO(10)$ but by an algebraic subgroup $H \subseteq G$. 
To describe the action of $H$, we first note that the involution $i$ commutes with the action of $H$, so that the quotient 
$(G(3,U^+)\times G(4,U^-))/H$ gives the isomorphism classes of smooth curves with an involution $i$.
Then the matrices in $SO(10)$ which act on the product subvariety are those which commute with the involution $i$ on a linear space $\p(U_{16})$.

As noted in \eqref{pmeigenspace}, we have an eigenspace decomposition
$$
U_{16}=U^+\oplus U^-
$$
for the action of $i$.
Since for any $h\in H$ and $s\in U^+$ (resp. $s\in U^-$)
$$
i.h(s)=h.i(s)=h(s)
$$
it follows that $U^+$ (resp. $U^-$) are (projective) $H$-modules.

By Corollary \ref{generic}, we know that a generic curve $C'\in \cR_{3,2}$ is non-tetragonal. Hence, the moduli space
$\cR_{3,2}$ does not lie in the indeterminacy locus of the birational map
$$
\cM_7  \rar  G(7,U^{16})/SO(10).
$$
Hence this birational map restricts to a generically injective rational map
$$
\cR_{3,2}\sta{\psi} {\rar}  G(7,U^{16})/SO(10).
$$
Corresponding to a non-tetragonal curve $(C'\rar C)\in \cR_{3,2}$ (which is the generic situation, by Corollary \ref{generic}) we can associate a point in $G(3,U^+)\times G(4,U^-)$ according to the decomposition of the canonical space in \eqref{eigenspaces}.
Hence the image of $\psi$ maps to the product space
$$
\cR_{3,2}\sta{\psi'}{\rar} (G(3,U^+)\times G(4,U^-))/H,
$$
and this map is generically injective.

To see that $\psi'$ is birational, given a generic point
$\p^6 \in G(3,U^+)\times G(4,U^-)$ we first know by \cite{Mukai3} that the intersection
$C'=\p^6\cap X_{10}$ lies in $\cM_7$. Now by Proposition \ref{cornalba2}, $C'$ lies in the singular locus $\cS\subset \cM_7$, since it has a nontrivial involution. This implies that the inverse image of  $(G(3,U^+)\times G(4,U^-))/H$ under $\psi$
in $\cM_7$ is a subset in the singular locus $\cS\subset \cM_7$ and containing a dense open subset of $\cR_{3,2}$.
But again by Proposition \ref{cornalba2} since $\cR_{3,2}$ is a maximal component in $\cS$, the inverse image has to be dense in $\cR_{3,2}$.

This proves the birational equivalence
\begin{equation}
\cR_{3,2}\sim (G(3,U^+)\times G(4,U^-))/H.
\end{equation}
\end{proof}

\begin{corollary}
The moduli space $\cR_{3,2}$ is a unirational variety.
\end{corollary}
\begin{proof}
Since a Grassmannian variety is a rational variety, it follows that the product space $G(3,U^+)\times G(4,U^-)$ is also
a rational variety. Using the description in \eqref{birat}, it follows that the moduli space $\cR_{3,2}$ is a unirational
variety.
\end{proof}

The birational model in \eqref{birat} should also be compatible with the projection $\cR_{3,2}\rar \cM_3$. We have been unable to determine $H$ explicitly and
we pose the following question:

\begin{Question}\label{question}:
 Determine the subgroup $H$ and the $H$-(projective) representations $U^+$ and $U^-$ explicitly.
\end{Question}

Notice that we have the spinor representation
$$
\phi(10):Spin(10)\rar Aut(U_{16})
$$
which gives the $SO(10)=\f{Spin(10)}{\pm 1}$-action on $P(U_{16})$, considered in \cite{Mukai3}.
It may be possible to study further via the spinor representation restricted to the various subgroups of $S)(10)$. 

Recall that the Spin group $Spin(2n)$ has two inequivalent irreducible
spinor representations of dimension $2^{n-1}$, denoted by $U^{\pm}_{2^{n-1}}$ and the Spin group $Spin(2n+1)$
has one irreducible spinor representation of dimension $2^n$, denoted by $U_{2^n}$, for example see \cite[\S 20.2 and Exercise 20.40,p.311]{Fulton} for these facts.
The above spinor representation $\phi(10)$
restricts on $Spin(8)$ to the automorphisms of $U_8^+\oplus U_8^-$; the sum of the inequivalent 
two irreducible spinor representations of $Spin(8)$ of dimension $8$. We know by previous Lemma \ref{SO} that $U^+,U^-$ are $H$-modules.
If $H=SO(8)$ then $U^+= U_8^+ \m{ or }U_8^-$ and $U^-= U_8^- \m{ or }U_8^+$.
But $\m{dim }SO(8)$ is $28$. Hence $\m{dim }\cR_{3,2}$ is not birationally equivalent to
$(G(3,U^+)\times G(4,U^-))/SO(8)$. Hence $H\neq SO(8)$.

The spinor representation $\phi(10)$ restricts to
two copies of the spinor representation $U_{8}$ of $Spin(7)$ of dimension $8$.

If $H=SO(7)$ then $U^+=U_8,U^-=U_8$ and $SO(7)$ acts on
$G(3,U^+)\times G(4,U^-)$. Since $\m{dim }SO(7)$ is $21$, we have dimension of $(G(3,U^+)\times G(4,U^-))/H$
equal to $10$ which is the same as $\m{dim }\cR_{3,2}=10$.

Similarly, $SO(6)$ has two inequivalent spinor representations $U_{4}^+$ and
$U_{4}^-$, and $SO(4)$ has two inequivalent spinor representations $U^+_2$ and $U^-_2$.
If $H=SO(4)\times SO(6)$ then $U^\pm=U_2^\pm\otimes U_4^\pm$.
Since dimension of $SO(4)\times SO(6)$ is $6+15=21$, we again get equality of the dimensions of $(G(3,U^+)\times G(4,U^-))/H$ and that of $\cR_{3,2}$.

Of course, the above discussion gives only some possibilities and is not really a proof. 

We leave it to the reader to continue
this interesting discussion.

In the next section we will investigate the question of rationality of $\cR_{3,2}$. Since we have been unable to
describe the subgroup $H$ and the $H$-subspaces $U^+$ and $U^-$ explicitly, we will look for another description of $\cR_{3,2}$ which we hope will lead to an answer on the rationality question.


\section{Rationality of the moduli space $\cR_{3,2}$}


In this section, we will prove the rationality of the moduli space $\cR_{3,2}$, via another description and using known
results on rationality of moduli space of curves of genus $g$ with $n$ marked points $\cM_{g,n}$, for small $g$ and $n$.
Recall that rationality of moduli spaces of curves with marked points has attracted wide interest and we list some recent results by  Katsylo, Dolgachev, Casnati-Fontanari (\cite{Katsylo}, \cite{Dolgachev}, \cite{Casnati}).
Rationality of some moduli spaces of double covers have also been obtained by Bardelli-Del Centina, Izadi-Lo Giudice-Sankaran (\cite{B-dC}, \cite{Izadi}).
To our knowledge the moduli space $\cR_{g,b}$, for $b>0$, introduced in \cite{BCV} have not been looked into.
We illustrate the case when $g=3$ and $b=2$ and expect that the results can be extended to some other cases as well.

Our main observation is the following:

\begin{lemma}\label{rationalJ}
The moduli space $\cR_{3,2}$ is birational to a $\p^1$-bundle over the universal Picard scheme $\m{Pic}^2_{\cM_3}$ which parametrises degree $2$ line bundles, over (an open subset of) the moduli space $\cM_3$.
\end{lemma}
\begin{proof}
Recall that $\cR_{3,2}$ parametrises triples $(C,L,B)$ of data: $C$ is a connected smooth projective curve of genus $3$,
$L$ is a line bundle of degree $2$ on $C$ and $B$ is a general divisor (consisting of distinct points) in the complete linear system $|L^2|$. Let $\cC\rar \cM_3$ and $\cJ\rar \cM_3$ denote the universal curve and universal Jacobian, and which exist over some open subset of the moduli space $\cM_3$. Consider the universal Picard variety
$\m{Pic}^2_{\cM_3}\rar \cM_3$. This family  parametrises line bundles of degree $2$ over a curve $C\in \cM_3$.

In particular the variety $\m{Pic}^2_{\cM_3}$ is the moduli space of pairs $(C,L)$ of the following data: $C$ is a connected smooth projective curve of genus $3$ and $L$ is a line bundle of degree $2$ on $C$. Since there is no universal Poincar\'e line bundle $\cL\rar \cC\times_{\cM_3} \m{Pic}^2_{\cM_3}$, we consider the universal Poincar\'e line bundle $\cL\rar \cC\times_{\cM_{3,1}} \m{Pic}^2_{\cM_{3,1}}$. Here $\cM_{3,1}$ denotes the moduli space of genus 3 curves with one marked point and $\cC\rar \cM_{3,1}$ is the universal curve with a section.    Consider the projections
$$
\xymatrix{
 & \cC \times_{\cM_{3,1}}  \m{Pic}^2_{\cM_{3,1}} \ar@{->}[dl] \ar@{->}[dr] & \\
 \cC &   &  \m{Pic}^2_{\cM_{3,1}}                                   }
$$
which are denoted by $p$ and $q$ respectively.

 Now look at the map
$$
h:\m{Pic}^2_{\cM_{3,1}}\rar \m{Pic}^2_{\cM_{3}}. 
$$
This is the same as the pullback of $\m{Pic}^2_{\cM_{3}}\rar \cM_3$ via the morphism given by forgetting the marked point:
$$
\cM_{3,1}\rar \cM_3.
$$
Consider the direct image sheaf $\cF:= (h_*\circ q_*)\cL^2$ on $\m{Pic}^2_{\cM_{3}}$.
The fibres of the sheaf $\cF$ are $H^0(C\times C,p_1^*L^2)=H^0(C,L^2)$, where $p_1:C\times C\rar C$ is the first projection.
 By Riemann-Roch theorem, there is an open subset $U\subset \m{Pic}^2_{\cM_3}$ such that the fibres of the sheaf $\cF$ are equi-dimensional and have dimension equal to $2$. Hence, by semi-continuity, $\cF$ forms a vector bundle over $U$ and denote its dual by $\cF^*$ over $U$.
Consider the projectivization $\cG:=\p(\cF^*)\rar \m{Pic}^2_{\cM_3}$. Then this is a $\p^1$-bundle over the open subset $U$ of $ \m{Pic}^2_{\cM_3}$ whose fibres are identified with the linear system $|L^2|$. Hence $\cG$ parametrises triples $(C,L,B)$ such that $B\in |L^2|$. Consider the open subset $U'$ of $\cG$ such that the points of $U'$ correspond to triples $(C,L,B)$ and the points in $B$ are distinct. Then $U'$ is precisely the moduli space $\cR_{3,2}$.
In other words, $\cR_{3,2}$ is birationally isomorphic to $\cG$.
\end{proof}

\begin{corollary}\label{rationalF}
Suppose the universal Picard scheme $\m{Pic}^2_{\cM_3}$ is a rational variety. Then the variety $\cG$ is also a rational variety.
\end{corollary}
\begin{proof}
Above we showed over an open subset of $\m{Pic}^2_{\cM_3}$ that $\cG$ is a $\p^1$-bundle  which is the projectivisation of a rank two vector bundle. 
This implies that $\cG$ is rational.
\end{proof}

\begin{corollary}\label{doublerational}
The moduli space $\cR_{3,2}$ is a rational variety, if $\m{Pic}^2_{\cM_3}$ is rational.
\end{corollary}
\begin{proof}
This follows from the birational isomorphism $\cR_{3,2}\sim \cG$ shown in the proof of Lemma \ref{rationalJ} and using
the rationality of $\cG$ shown in Corollary \ref{rationalF}.
\end{proof}

\begin{remark}\label{remverra}
It is mentioned by Verra in \cite[Introduction]{Verra2} that the universal abelian variety over $\cM_3$ is rational, using the results in \cite{Casnati}. Although we do not have a proof of this, Verra \cite{Verra3} has communicated to us that this is highly probable. 
\end{remark}

\begin{remark} For other values $g=4,5,6,8,10$ and some small values of $b$ depending on $g$, 
similar arguments and proof are likely to prove the rationality of $\cR_{g,b}$. This may follow from the rationality results for moduli spaces
of pointed curves in \cite{Casnati, BCF}.
\end{remark}


\section{Chow--K\"unneth decomposition for an open subset of $\cR_{3,2}$}

In this section, we want to conclude that there is an open subset of $\cR_{3,2}$  which has a Chow--K\"unneth decomposition. See similar results in \cite{Iy-Mu} for open subsets of moduli space of curves of small genus $g\leq 8$.
Recall that this was proved in \cite{Iy-Mu}, via realizing the open subsets as group quotients of open subsets in homogeneous spaces. The key point used was that the homogeneous spaces have only algebraic cohomology and hence orthogonal projectors equivariant for the group action could be constructed. All those results could also be applied to the variety $\cR_{3,2}$.

\begin{corollary}\label{CKopen}
There is an open subset of the moduli space $\cR_{3,2}$ which admits a Chow--K\"unneth decomposition in the sense of definition \ref{CK-def}.
\end{corollary}
\begin{proof}
We use Lemma \ref{SO} and the birational equivalence
\begin{equation}\label{birat}
\cR_{3,2}\sim (G(3,U^+)\times G(4,U^-))/H.
\end{equation}
to conclude that there is an open subset $U\subset \cR_{3,2}$ which is isomorphic to an open subset $U'$ of a group quotient of the homogeneous space $ G(3,U^+)\times G(4,U^-)$. Since the product of Grassmannian varieties has only algebraic cohomology, it has a Chow--K\"unneth decomposition, by Proposition \ref{simpleprojectors}. The orthogonal projectors for $G(3,U^+)\times G(4,U^-)$ can be lifted in the rational equivariant Chow group of the product of Grassmannians, for the action of the group $H$ (see \cite[Lemma 5.2]{Iy-Mu}). These (equivariant) orthogonal projectors correspond to orthogonal projectors for the bottom weight cohomology for $U'$ (the proof is similar to \cite[Corollary 5.9]{Iy-Mu} and we do not repeat them here). This precisely gives a Chow--K\"unneth decomposition for $U'$. 

\end{proof}




\begin{thebibliography}{AAAAA}

\bibitem[BCF]{BCF} E. Ballico, G. Casnati, C. Fontanari, {\em On the birational geometry of moduli spaces of pointed curves}, ArXiv:math/0701475 (2007).

\bibitem[BCV]{BCV} F. Bardelli, C. Ciliberto, A. Verra, {\em Curves of minimal genus on a general abelian variety.}  Compositio Math.  96  (1995),  no. \textbf{2}, 115--147.

\bibitem[B-dC]{B-dC} F. Bardelli, A. Del Centina, {\em Bielliptic curves of genus three: canonical models and moduli space. } Indag. Math. (N.S.) 10 (1999), no. \textbf{2}, 183--190.


\bibitem[Be]{Behrend} K. Behrend,  {\em On the de Rham cohomology of differential and algebraic stacks.}  Adv. Math.  198  (2005),  no. \textbf{2}, 583--622.



\bibitem[Bg-To]{tommasi} J. Bergstr\"om, O. Tommasi, {\em The rational cohomology of $\overline{\cM_4}$.}  Math. Ann.  338  (2007),  no. \textbf{1}, 207--239.

\bibitem[Ca-Fo]{Casnati} G. Casnati, C. Fontanari, {\em On the rationality of moduli spaces of pointed curves.} J. Lond. Math. Soc. (2) 75 (2007), no. \textbf{3}, 582--596.

\bibitem[Ch-Rn]{chen} M.C. Chang, Z. Ran, {\em Unirationality of the moduli spaces of curves of genus $11,$ $13$ (and $12$).} Invent. Math. 76 (1984), no. \textbf{1}, 41--54.

\bibitem[Co]{Cornalba2} M. Cornalba, {\em On the locus of curves with automorphisms.}  (Italian summary)
Ann. Mat. Pura Appl. (\textbf{4}) 149 (1987), 135--151.


\bibitem[Co2]{cornalba} M. Cornalba, {\em Cohomology of moduli spaces of stable curves.} Proceedings of the International Congress of Mathematicians, Vol. II (Berlin, 1998).  Doc. Math.  1998,  Extra Vol. II, 249--257.



\bibitem[dA-Ml]{dA-Mu} P. del Angel, S. M\"uller--Stach, {\em
 Motives of uniruled $3$-folds}, Compositio Math. 112 (1998), no.
 \textbf{1},
 1--16.

\bibitem[dA-Ml2]{dA-Mu2} P. del Angel, S. M\"uller-Stach, {\em On
 Chow motives of 3-folds}, Trans. Amer. Math. Soc. 352 (2000), no.
 \textbf{4}, 1623--1633.


\bibitem[De-Mu]{De-Mu} Ch. Deninger, J. Murre, {\em Motivic
 decomposition of abelian schemes and the Fourier transform}, J. Reine
 Angew. Math.
 \textbf{422} (1991), 201--219.

\bibitem[Del-Mm]{Del-Mu}
P. Deligne, D. Mumford, {\em The irreducibility of the space of curves
 of given genus}, Inst. Hautes \'Etudes Sci. Publ. Math. No.
 \textbf{36}
 1969 75--109.

\bibitem[Do]{Dolgachev} I. Dolgachev, {\em Rationality of $\cR\sb 2$ and $\cR\sb 3$.} Pure Appl. Math. Q. 4 (2008), no. \textbf{2}, part 1, 501--508.



\bibitem[Fu2]{Fulton} W. Fulton, J. Harris, {\em Representation theory. A first course.} Graduate Texts in Mathematics, 129. Readings in Mathematics. Springer-Verlag, New York, 1991. xvi+551 pp.

\bibitem[Ge]{getzler} E. Getzler, {\em Topological recursion relations in genus $2$.}  Integrable systems and algebraic geometry (Kobe/Kyoto, 1997),  73--106, World Sci. Publ., River Edge, NJ, 1998.

\bibitem[Ge]{getzler-looijenga} E. Getzler, E. Looijenga, {\em  The Hodge polynomial of $\ov\cM_{3,1}$}  arXiv:math.AG/9910174, 1999.


\bibitem[Gi]{Gi} H. Gillet,  {\em  Intersection theory on algebraic stacks and $Q$-varieties}, Proceedings of the
Luminy conference on algebraic $K$-theory (Luminy, 1983).
J. Pure Appl. Algebra \textbf{34} (1984), 193--240.



\bibitem[Go-Mu]{Go-Mu} B. Gordon, J. P. Murre,  {\em Chow motives of
 elliptic modular threefolds}, J. Reine Angew. Math. \textbf{514}
 (1999),
 145--164.

\bibitem[GHM1]{GHM1} B. Gordon, M. Hanamura, J. P. Murre,  {\em
 Relative Chow-K\"unneth projectors for modular varieties}, J. Reine
 Angew.
 Math.  \textbf{558}  (2003), 1--14.

\bibitem[GHM2]{GHM2} B. Gordon, M. Hanamura, J. P. Murre,  {\em
 Absolute Chow-K\"unneth projectors for modular varieties}, J. Reine
 Angew. Math.
 \textbf{580} (2005), 139--155.



\bibitem[Ha]{Ha} J. Harer, {\em Stability of the homology of the mapping class groups of orientable surfaces.}  Ann. of Math. (2)  121  (1985),  no. \textbf{2}, 215--249.

\bibitem[Hs]{Harris} J. Harris, {\em Algebraic geometry. A first course.} Corrected reprint of the 1992 original. Graduate Texts in Mathematics, 133. Springer-Verlag, New York, 1995. xx+328.

\bibitem[Hn]{Arbarello} R. Hartshorne, {\em Algebraic geometry.} Graduate Texts in Mathematics, No. \textbf{52}. Springer-Verlag, New York-Heidelberg, 1977. xvi+496 pp.

\bibitem [Iv]{ivanov} N.V. Ivanov, {\em Complexes of curves and Teichm\"uller modular groups.} (Russian)  Uspekhi Mat. Nauk  42  (1987),  no. \textbf{3} (255), 49--91.


\bibitem [Iv2]{ivanov2} N.V. Ivanov, {\em On the homology stability for Teichm\"uller modular groups: closed surfaces and twisted coefficients.}  Mapping class groups and moduli spaces of Riemann surfaces (G\"ottingen, 1991/Seattle, WA, 1991),  149--194, Contemp. Math., \textbf{150}, Amer. Math. Soc., Providence, RI, 1993.

\bibitem[Iy]{Iy} J. N. Iyer,  {\em Murre's conjectures and explicit
 Chow K\"unneth projectors for varieties with a nef tangent bundle}, Trans. of Amer. Math. Soc. 361 (2009), no. \textbf{3}, 1667--1681.

\bibitem[Iy-Ml]{Iy-Mu} J. N. Iyer, S. M\"uller-Stach, {\em Chow--K\"unneth decomposition for some moduli spaces}. Documenta Mathematica, \textbf{14}, 2009, 1-18.

\bibitem[Iz-G-S]{Izadi} E. Izadi, M. Lo Giudice, G. Sankaran, {\em The moduli space of tale double covers of genus 5 curves is unirational.} Pacific J. Math. 239 (2009), no. \textbf{1}, 39--52.

 
\bibitem[Ka]{Katsylo} P. Katsylo, {\em Rationality of the moduli variety of curves of genus $3$.} Comment. Math. Helv. 71 (1996), no. \textbf{4}, 507--524.


\bibitem [Lo1]{Lo} E. Looijenga, {\em Stable cohomology of the mapping class group with symplectic coefficients and of the universal Abel-Jacobi map.}  J. Algebraic Geom.  5  (1996),  no. \textbf{1}, 135--150.


\bibitem [Lo2]{looijenga2} E. Looijenga, {\em Cohomology of ${\cM_3}$ and
$\cM_{3,1}$.}  Mapping class groups and moduli spaces of Riemann surfaces (G\"ottingen, 1991/Seattle, WA, 1991),  205--228, Contemp. Math., \textbf{150}, Amer. Math. Soc., Providence, RI, 1993.

\bibitem[Ma-We]{Ma-We}
I. Madsen, M. Weiss, {\em The stable moduli space of Riemann surfaces: Mumford's conjecture.}  Ann. of Math. (2)  165  (2007),  no. \textbf{3}, 843--941.


\bibitem[MWYK]{MM} A. Miller, S. M\"uller-Stach, S. Wortmann, Y.H.
 Yang,  K. Zuo. {\em Chow-K\"unneth decomposition for universal
 families over
 Picard modular surfaces}, arXiv:math.AG/0505017, 2005.

\bibitem[Mk]{Mukai3} S. Mukai, {\em Curves and symmetric spaces. I.}
  Amer. J. Math.  117  (1995),  no. \textbf{6}, 1627--1644.


\bibitem[Mk2]{Mukai6} S. Mukai, {\em Curves and symmetric spaces.}  Proc. Japan Acad. Ser. A Math. Sci.  68  (1992),  no. \textbf{1}, 7--10.

\bibitem[Mm]{Mumford}
D. Mumford,  {\em Towards an Enumerative Geometry of the
Moduli Space of Curves}, Arithmetic and geometry, Vol. II, 271--328, Progr.
Math., $\bf{36}$, Birkh\"auser Boston, Boston, MA, 1983.

\bibitem[Mu1]{Mu1} J. P. Murre, {\em On the motive of an algebraic
 surface}, J. Reine Angew. Math.  \textbf{409}  (1990), 190--204.

\bibitem[Mu2]{Mu2} J. P. Murre, {\em On a conjectural filtration on
 the Chow groups of an algebraic variety. I. The general conjectures and
 some examples}, Indag. Math. (N.S.)  4  (1993),  no. \textbf{2},
 177--188.

\bibitem[Mu3]{Mu3} J. P. Murre,  {\em On a conjectural filtration on
 the Chow groups of an algebraic variety. II.
Verification of the conjectures for threefolds which are the product
 on a surface and a curve},  Indag. Math. (N.S.)  4  (1993),  no.
 \textbf{2}, 189--201.

\bibitem[Ra]{Ramanan} S. Ramanan, {\em Ample divisors on abelian surfaces.}  Proc. London Math. Soc. (3)  51  (1985),  no. \textbf{2}, 231--245.

\bibitem[Sc]{Sc}  A. J. Scholl,  {\em Classical motives}, Motives
 (Seattle, WA, 1991),  163--187, Proc. Sympos. Pure Math., \textbf{55},
 Part 1, Amer. Math. Soc., Providence, RI, 1994.

\bibitem[Se]{Sernesi} E. Sernesi, {\em Unirationality of the variety of moduli of curves of genus twelve.} (Italian) Ann. Scuola Norm. Sup. Pisa Cl. Sci. (4) 8 (1981), no. \textbf{3}, 405--439.

\bibitem[Sh]{Sh} A. M. Shermenev,  {\em The motive of an abelian
 variety}, Funct. Analysis, \textbf{8} (1974), 55--61.

\bibitem[Ve]{Verra} A. Verra, {\em The unirationality of the moduli spaces of curves of genus 14 or lower.} Compos. Math. 141 (2005), no. \textbf{6}, 1425--1444.

\bibitem[Ve2]{Verra2} A. Verra, {\em On the universal principally polarised abelian variety of dimension 4.} Curves and abelian varieties, 253--274, Contemp. Math., \textbf{465}, Amer. Math. Soc., Providence, RI, 2008. 

\bibitem[Ve3]{Verra3} A. Verra, {\em Letter to the first author}, dated 9 August 2009.






\end{thebibliography}
\end{document}